\documentclass[a4paper,11pt,reqno]{amsart}

\usepackage{tikz, subfigure}
\usetikzlibrary{arrows}

\usepackage{hyperref}

\numberwithin{equation}{section}

\newtheorem{theorem}{Theorem}[section]
\newtheorem{lemma}[theorem]{Lemma}
\newtheorem{defin}[theorem]{Definition}
\newtheorem{prop}[theorem]{Proposition}
\newtheorem{remark}[theorem]{Remark}
\newtheorem{conjecture}[theorem]{Conjecture}
\newtheorem{corollary}[theorem]{Corollary}

\def\la{\lambda}
\def\R{\mathbb{R}}

\def\Om{\Omega}
\def\C{\mathcal{C}}
\def\mobile{M}

\def\hull{\mathrm{hull}}
\def\ho{\Om_{\text{\tiny HO}}}

\def\A{{\mathcal A}}
\def\diam{\mathrm{diam}\,}

\def\eps{\varepsilon}
\def\E{\mathcal{E}}
\def\th{\vartheta}

\def\betafk{\beta_{\text{\tiny FK}}}
\def\betaks{\beta_{\text{\tiny KS}}}

\title[Convex combinations of low eigenvalues]{Convex combinations of low eigenvalues,\\ Fraenkel asymmetries and attainable sets}

\author{Dario Mazzoleni and Davide Zucco}

\address{Dario Mazzoleni, Universit\`a degli Studi di Torino, Torino, Italy}
\email{dmazzole@unito.it}

\address{Davide Zucco, Scuola Internazionale Superiore di Studi Avanzati, Trieste, Italy}
\email{davide.zucco@sissa.it}

\begin{document}

\begin{abstract}
We consider the problem of minimizing convex combinations of the first two eigenvalues of the Dirichlet-Laplacian among open sets of $\R^N$ of fixed measure. We show that, by purely elementary arguments, based on the minimality condition, it is possible to obtain informations on the \emph{geometry} of the minimizers of convex combinations: we study, in particular, when these minimizers are no longer convex, and the optimality of balls. As an application of our results we study the boundary of the attainable set for the Dirichlet spectrum. 

Our techniques involve symmetry results \emph{\`a la} Serrin, explicit constants in quantitative inequalities,  as well as a purely geometrical problem: the minimization of the Fraenkel 2-asymmetry among \emph{convex} sets of fixed measure.
\end{abstract}

\maketitle

\paragraph{\textbf{Keywords:}}Eigenvalues, Dirichlet Laplacian, Fraenkel asymmetry, attainable set.

\smallskip
\paragraph{\textbf{Mathematics Subject Classification:}} 47A75, 49G05, 49Q10, 49R05.

\section{Introduction}

Spectral optimization problems have received a lot of attention in the last years, with a particular emphasis to extremum problems for eigenvalues of elliptic operators (see the books \cite{BB,H,HP}). A typical problem consists in the minimization of a functional defined in terms of the eigenvalues of the Laplace operator among sets of fixed measure. Here to simplify the exposition we will always consider the measure constraint equal one. The first issue for this kind of problems concerns the \emph{existence} of an optimal shape: a result proved in the 1990s by Buttazzo and Dal Maso \cite{BDM} is even now a cornerstone of the matter, and, for a large class of functionals, it ensures the existence of a solution in the class of quasi-open sets of fixed measure (\emph{a priori} contained into a given box, which provides the necessary compactness to prove existence). 
Moreover, the \emph{regularity} of an optimal shape is a highly difficult problem and a general regularity theory is nowadays not available: even a proof which guarantees that an optimal shape is open, and not merely quasi-open, is far from being trivial, see~\cite{bmpv}.
Another important point consists in proving some \emph{geometric} properties of optimal shapes, such as connectedness, convexity, symmetry with respect to some axis, and so on. By the way, only for few special functionals optimal shapes are explicitly known: classical examples are the lowest eigenvalues of the Dirichlet-Laplacian.
We recall that, for a given integer $N\geq 2$ and an open set $\Omega\subset \R^N$ with finite measure, the \emph{first} and \emph{second eigenvalues} of the Dirichlet-Laplacian can be defined as
\begin{equation*}
\la_1(\Om):=\min_{u\in H^1_0(\Om)\setminus\{0\}}\frac{\int_{\Omega} |\nabla u(x)|^2dx}{\int_{\Omega} |u(x)|^2dx}, \qquad \la_2(\Om):=\min_{\stackrel{u\in H^1_0(\Om)\setminus\{0\}}{ \int_\Omega u u_1=0}}\frac{\int_{\Omega} |\nabla u(x)|^2dx}{\int_{\Omega} |u(x)|^2dx},
\end{equation*}
where these minima are attained, respectively, by the \emph{first} and \emph{second eigenfunctions} $u_1$ and $u_2$ (which are unique, up to a multiplicative constant). 

The interest in the minimization of the first eigenvalue goes back to a conjecture due to Lord Rayleigh in 1877, then proved by Faber and Krahn in the 1920s. The \emph{Faber-Krahn inequality} asserts that of all open sets of fixed measure, the ball has the minimum first eigenvalue:
in formula, for every open set $\Om\subset \R^N$ with unit measure 
\begin{equation}\label{FK}
\lambda_1(\Omega)\ge\lambda_1(B)=\omega_N^{2/N} j_{N/2-1}^2,
\end{equation}
where $\omega_N$ denotes the measure of the ball in $\R^N$ with unit radius, $j_{n}$ the first positive zero of the Bessel function $J_n$, and $B$ the open ball of unit measure in $\R^N$. Equality in \eqref{FK} holds if and only if $\Omega$ is that ball (up to sets of capacity zero).  The same issue for the second eigenvalue is known as the \emph{Krahn-Szeg\"o inequality}, which asserts that two disjoint open balls of half measure each are the unique (up to sets of capacity zero) minimizer, namely for every open set $\Om\subset \R^N$ with unit measure
\begin{equation}\label{KS} 
\lambda_2(\Omega)\ge \la_2(B_-\cup B_+)=2^{2/N} \lambda_1(B)=(2\omega_N)^{2/N} j_{N/2-1}^2,
\end{equation}
where $B_-\cup B_+$ is the union of two equal and disjoint open balls of half measure each, and equality in \eqref{KS} holds if and only if $\Omega=B_-\cup B_+$.

\smallskip

Starting with the important work of Keller and Wolf~\cite{KW}, there was a strong interest for \emph{convex combinations} of the first two eigenvalues of the Dirichlet-Laplacian, namely the functional $F_t$ defined, for every $t\in(0,1)$, as
\begin{equation}\label{convexcomb}
F_t(\Om):=t\la_1(\Om)+(1-t)\la_2(\Om),
\end{equation}
where $\Om\subset \R^N$ is an open set of finite measure.
Then, the corresponding spectral optimization problem writes as
\begin{equation}\label{P}
	\min{\left\{F_t(\Omega)\;:\;\text{$\Om\subset\R^N$, $\Omega$ open, $|\Om|=1$}\right\}}.
\end{equation}
The existence of a minimizer for this problem is now well understood and is guaranteed by a general theory recently developed in the works~\cite{Bucur,bmpv,mp}, all based on the above mentioned result \cite{BDM}, but with the new difficulty of working in the entire space $\R^N$. Notice that, all these results guarantee the existence of an optimal shape in the larger class of quasi-open sets, and only \emph{a posteriori} one proves that a minimizer of problem \eqref{P} is in fact \emph{open}, and so problem \eqref{P} is well-posed.
Moreover, in~\cite{im} it was proved that, for every $t\in(0,1)$, minimizers of \eqref{P} are \emph{connected} (more generally, this topological property was studied for minimizers of convex combinations of the first three eigenvalues). In two dimensions ($N=2$), some numerical computations on the shape of these minimizers appeared in \cite{KO}.
We sum up all these results in the following theorem.
 
\begin{theorem}
For every $t\in(0,1)$, there exists a minimizer in \eqref{P}. Moreover, every minimizer $\Omega_t$ is a connected set of finite perimeter with uniformly bounded diameter (depending only on the dimension $N$).
\end{theorem}

\smallskip

Our goal here is to show that, by purely elementary arguments essentially based on the minimality condition, it is possible to obtain interesting informations on the geometry of the minimizers for problem~\eqref{P}, and to recover some known results on the boundary of the attainable set for the Dirichlet spectrum (see \cite{AH2,ah,BNP,KW}). 

Notice that, if $t=1$ the convex combination \eqref{convexcomb} is minimized by the ball with unit measure (because of the Faber-Krahn inequality \eqref{FK}), while if $t=0$,  by two equal balls of half measure each (because of the Krahn-Szeg\"o inequality \eqref{KS}). Therefore, as $t$ moves from $1$ to $0$, one expects the shape of a minimizer $\Om_t$ deforming from a ball of unit measure to two balls of half measure each; in particular, it is natural to conjecture that at some value of $t$ the convexity of all the minimizers in \eqref{P} is lost (as was numerically observed in \cite{KO}, in two dimensions, the critical value for $t$ is expected to be $1/2$).
This question also appeared in \cite{H} as the Open Problem 21. We give a first answer to this question, though non-optimal.
All the results of this paper, unless otherwise specified, will hold in every dimension $N\geq 2$.

\begin{theorem}\label{teo.1}
There exists a threshold $T>0$ such that, for all $t\in (0,T)$, every minimizer in~\eqref{P} is no longer convex.
\end{theorem}

We provide a \emph{quantitative} proof of this theorem, namely we explicitly construct  the threshold $T$ \emph{via} the eigenvalues of the Dirichlet-Laplacian \eqref{explicit} (cf. with \cite[Section 5.3]{bbv} where a similar question was analyzed in a different context, and whose strategy of proof could be adapted to prove Theorem~\ref{teo.1}, however without getting an explicit value for $T$).
To be more concrete, in two dimensions, we provide a numerical lower bound on $T$ using a quantitative Krahn-Szeg\"o inequality involving the so-called Fraenkel 2-asymmetry. 
Therefore, we are naturally led to consider a purely geometrical problem, which is probably the most innovative part of the paper: the minimization of the Fraenkel 2-asymmetry among \emph{convex} sets of given area.
We show that the \emph{mobile}, i.e., the intersection of the convex hull of two tangent balls with a strip, see Definition \ref{mobile},
is the unique minimizer satisfying an isoperimetric inequality for the Fraenkel 2-asymmetry \eqref{isofras}. An explicit value for the constant in the quantitative Krahn-Szeg\"o inequality will be also needed. This opens a new area of application for quantitative inequalities.

As second question, we analyze the optimality of a special convex set: the ball, generalizing a result from~\cite{KW}. 
\begin{theorem}\label{teo.2}
For all $t\in (0,1)$ the ball $B$ is never a minimizer in \eqref{P}.
\end{theorem}

We prove more generally that the second eigenvalue of a minimizer in \eqref{P} has to be simple and, as a consequence of the multiplicity of the second eigenvalue over balls, we immediately get the result in Theorem \ref{teo.2}.
The proof of the simplicity of the second eigenvalue relies on some ideas developed in \cite{H} and \cite{bbh}, with the help of a classical symmetry result due to Serrin \cite{serrin} (see also \cite{fg}). 

\smallskip

As an application of our results, we show how to get informations on the shape of the \emph{attainable set}, namely the subset of the plane described by the range of the first two eigenvalues of the Dirichlet-Laplacian
\begin{equation}\label{attainable}
	\E:=\left\{(x,y)\in\R^2\;:\;x=\la_1(\Om),\; y=\la_2(\Om),\;\Om\subset\R^N, \;\mbox{$\Om$ open},\;|\Om|=1\right\}.
\end{equation}
This set was introduced in \cite{KW}, and then deeply studied in \cite{bbf} (see also \cite{AH2, ah, BNP}), 
where several geometrical properties of $\E$ were discussed (the closedness of $\E$ is important for the existence of optimal shapes for non monotone functionals see \cite{bbf}). 

The link between problem~\eqref{P} and the set $\E$ is the following: for a fixed $t\in(0,1)$ 
a minimizer $\Om_t$ in \eqref{P} corresponds to the first point of $\E$ of coordinates $(\la_1(\Om_t),\la_2(\Om_t))$ that we reach with a line $tx+(1-t)y=a$ increasing the value $a$, that is $P_{\Om_t}:=(\la_1(\Om_t),\la_2(\Om_t))$ is one of the intersection points of the tangent line to $\E$ with slope $t/(t-1)$. In particular, if $t=1$ the tangent line $x=\la_1(B)$ has a \emph{unique} intersection point corresponding to the ball $B$ (because of the Faber-Krahn inequality \eqref{FK}), while if $t=0$, the tangent line $y=\la_2(\Theta)$ has a \emph{unique} intersection point corresponding to the two balls $B_-\cup B_+$ (because of the Krahn-Szeg\"o inequality \eqref{KS}).

Therefore, in Theorem~\ref{teo.tangent}, we will present a new strategy for studying the asymptotic behavior of the boundary of $\E$ near the points corresponding to $B$ and $B_-\cup B_+$, extending to all dimensions a result proved in \cite{KW} only in two dimensions, and recovering the result proved in \cite{BNP}. 
Indeed, according to \cite{AH2, ah, BNP,KW}, the common strategy to study the asymptotic behavior of $\partial \E$ consists of two steps: the construction, through a parameter $\epsilon>0$, of an explicit perturbation of the set corresponding to the limit point on the boundary of $\E$ (i.e., a ball or two balls) and then the computation of the corresponding limit as $\epsilon\to 0$. Here instead, to compute such a limit, we rely on the minimality condition of the minimizers of convex combinations \eqref{minimality} without any explicit construction. Moreover, looking at the boundary of the attainable set through convex combinations is very useful, since it works in all dimensions, and can be applied to other attainable sets with different constraints. 

We suspect that to properly understand the boundary behavior of the attainable set, one has to carefully analyze problem \eqref{P}. For this reason we restate the long-standing conjecture about the convexity of the attainable set in the language of the minimizers of convex combinations of the lowest Dirichlet eigenvalues. Notice that the link of problem \eqref{P} with the attainable set is not new, in fact it was used in \cite{KW} (and more recently in \cite{AH2} and \cite{ah}) to draw numerically the boundary of the attainable set. However, up to our knowledge, it was never used to study analytic results.

\smallskip
The paper is organized as follows.
In Section~\ref{sec.fras} we derive an isoperimetric inequality for the Fraenkel 2-asymmetry. In Section~\ref{sec.comb} we prove Theorem~\ref{teo.1} and Theorem~ \ref{teo.2}. In Section~\ref{sec.atta} we show some applications to the attainable set. In Appendix~\ref{appendix} we compute explicit constants in quantitative inequalities.

\smallskip

\paragraph{\textbf{Notation}} Throughout the paper, we always denote by $\Om_t$ a minimizer of problem~\eqref{P} for $t\in(0,1)$, by $B$ the open ball of unit measure and by $B_-\cup B_+$ two open balls of half measure each, saving the particular notation $\Theta$ when their centers are on the $x$-axis and they are \emph{tangent} in the origin (i.e., their closures are touching in the origin). We use the symbol $\subset$ to denote the \emph{strict} inclusion between sets, while $\subseteq$ if the inclusion holds possibly with the equality, and we use $\triangle$ for the symmetric difference between sets. We write $\diam(\cdot)$ and $\hull(\cdot)$ to denote the diameter and the convex envelope of a set, respectively. We use $\approx$ to denote an approximate value for a constant and we always consider only three decimal digits.

\section{An isoperimetric inequality for the Fraenkel 2-asymmetry}\label{sec.fras}

Quantitative inequalities are refinements of isoperimetric inequalities: they measure how far a set is from the optimal shape in terms of the deviation of the functional to its minimum value (for a brief overview on the topic see Appendix~\ref{appendix}). It is then important to look at the right quantity that provides such a measure. 

The most well-known example is the so called \emph{Fraenkel asymmetry}, which is often used when balls are optimal in an isoperimetric inequality: for an open set $\Om\subset\R^N$ with unit measure, it is defined as
\begin{equation}\label{2fras1}
\mathcal{A}(\Om):=\min{\left\{{|\Om\triangle B|}\;:\;  \text{$B$ open ball, $|B|=1$}\right\}}.
\end{equation}

In this paper we will rely on the Krahn-Szeg\"o inequality~\eqref{KS} in a quantitative form, and, since in this case two disjoint balls of equal measure are the optimal set, a different asymmetry is needed.
According to~\cite{bp}, the \emph{Fraenkel 2-asymmetry} is, for an open set  $\Om\subset\R^N$ with unit measure, defined as
\begin{equation}\label{2fras}
\A_2(\Om):=\min{\left\{{|\Om \triangle (B_-\cup B_+)|}\;:\, \text{$B_-,B_+$ disjoint open balls, $|B_-|=|B_+|={1}/{2}$}\right\}}.
\end{equation}
We point out that if $\Omega$ and $E$ are two measurable sets of $\R^N$ of the same measure $|\Omega|=|E|$, then
\begin{equation}\label{equivalent}
{|\Om\triangle E|}=2{|\Om \setminus	E|}=2{|E\setminus \Omega|},
\end{equation}
and this allows to write \eqref{2fras1} and \eqref{2fras} in a slightly different way, choosing $E=B$ or $E=B_-\cup B_+$. The relation \eqref{equivalent} will be used several times in the sequel.

A quantitative Krahn-Szeg\"o inequality was proved by Brasco and Pratelli in \cite{bp}: they showed the existence of a constant $C_{KS}>0$ (depending on the dimension) such that for all open sets $\Om\subset\R^N$ of finite measure $|\Om|=|B_-\cup B_+|$,
\begin{equation}\label{QKS}
\frac{\la_2(\Om)}{\la_2(B_-\cup B_+)}-1\geq C_{KS}\A_2(\Om)^{2(N+1)}.
\end{equation} 
Notice that $C_{KS}$ can be explicitly computed, see Appendix~\ref{appendix}.
In two dimension a quantitative Faber-Krahn inequality proved by Bhattacharya in~\cite{bha}, allows to improve the exponent of~\eqref{QKS}, still with a constant that can be explicitly computed: more precisely, if $N=2$ there exists $\betaks>0$ such that for all open sets $\Om\subset\R^2$ of finite measure $|\Om|=|B_-\cup B_+|$,
\begin{equation}\label{QKSbha}
\frac{\la_2(\Om)}{\la_2(B_-\cup B_+)}-1\geq \betaks\A_2(\Om)^{9/2},
\end{equation} 
where the constant $\betaks$, according to Appendix~\ref{appendix}, can be chosen as
\begin{equation}\label{constantKS}
\betaks=\frac{1}{24 j_0^2}\,\frac{\pi^{3/2}}{(\pi+2+2\pi^{1/3})^{9/2}}\cdot 10^{-5}\approx 3.331 \cdot 10^{-11}.
\end{equation}

\smallskip
In this section we analyze the following purely geometrical problem: minimize the Fraenkel 2-asymmetry among \emph{convex} sets
\begin{equation}\label{min2fras}
\inf \{\A_2(\Omega) : \; \text{$\Omega\subset \R^N$, $\Omega$ open and convex, $|\Omega|=1$} \}.
\end{equation}
The reason for studying problem \eqref{min2fras} is to quantitatively answer to Theorem~\ref{teo.1}. We focus only on the two dimensional case, since in this case we are able to identify the unique minimizer for this problem, although some of the results that we prove hold in any dimension (see Remark~\ref{rem.general}). 

The condition for the balls $B_-,B_+$ to be disjoint in \eqref{2fras} prevents two balls to overlap, but not their closures to be tangent.
This occurs, in particular, whenever the set $\Omega$ is convex. 

\begin{lemma}\label{lemma.theta}
Let $\Omega\subset \R^2$ be a convex open set with unit measure. 
The minimum in the definition of the Fraenkel 2-asymmetry \eqref{2fras} is attained by two \emph{tangent} balls (i.e. those whose closures are touching in a point).
\end{lemma}
\begin{proof}
We prove that tangent balls are always better competitors than non-tangent balls. Let $B_-$ and $B_+$ be two admissible balls for $\eqref{2fras}$ which are non-tangent.  We may assume both $\Omega\cap B_-$ and $\Omega\cap B_+$ to be non-empty sets, otherwise the quantity $|\Omega\triangle (B_-\cup B_+)|$ can be decreased by any other couple of non-tangent balls satisfying this property. Therefore, the boundary of the convex envelope $\hull\big((\Omega\cap B_-)\cup(\Omega\cap B_+)\big)$  outside the balls $B_-$ and $B_+$ is  made exactly by two segments. By moving and rotating the coordinate system we may assume these segments to be onto the half lines $y=\pm mx$ with $x\geq 0$, or in the limiting case $y=\pm m$ with $x\geq 0$, for a suitable choice of the constant $m>0$. Possibly exchanging the role of $B_-$ and $B_+$, we can also suppose $B_-$ to be on the left with respect to $B_+$.  All these assumptions combined with the convexity of $\Omega$ guarantee the following fact: to every point $P_1=(x_1,y)\in \Om\cap B_-$ there exists a point (with the same ordinate) $P_2=(x_2,y)\in \Om\cap B_+$, and moreover, the whole segment $\overline{P_1P_2}\subseteq \Omega$. Therefore, if $\tau$ denotes the translation toward the right, mapping the ball $B_-$ into the (left) tangent ball $\tau(B_-)$ to $B_+$, this means that $\tau(\Omega\cap B_-)=\tau(\Omega)\cap \tau (B_-)\subseteq \Omega\cap \tau(B_-)$. In particular $|\Omega\cap B_-|=|\tau(\Omega\cap B_-)|\leq|\Omega\cap \tau(B_-)|$, and so passing to the complementary sets yields that $|(B_-\cup B_+)\setminus\Omega|\geq |(\tau(B_-)\cup B_+)\setminus\Omega|$, which, together with \eqref{equivalent} for $E=B_-\cup B_+$, implies that the functional to be minimized has been not increased on tangent balls.
We remark that the configuration of the two optimal balls is not always unique (one can think, for example to $\Om$ as a long rectangle), but also in this case, two tangent balls are one of the admissible optimal configuration, and we choose them for the next steps. 
\end{proof}

We are ready to analyze in detail problem~\eqref{min2fras}. Therefore, we introduce the following definition.

\begin{defin}\label{mobile}
We call \emph{mobile} the open convex set $\mobile$ given by the intersection of the stadium $\hull(\Theta)$ (centered in the origin) with the horizontal strip $\{(x,y)\in\R^2\,:\,-h\leq y\leq h\}$, where $h>0$ is chosen so that $|\mobile|=1$ (see Figure \ref{figure}).
\end{defin}

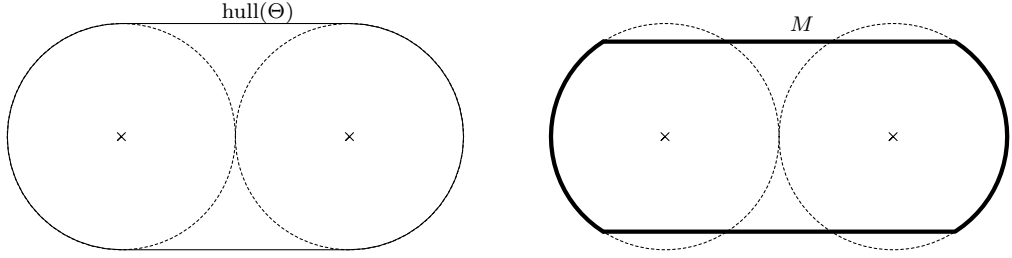
\begin{figure}[t]
\centering
\subfigure{
\begin{tikzpicture}[scale=1.5, line cap=round,line join=round,>=triangle 45,x=1.0cm,y=1.0cm]
\clip(-2.21,-1.35) rectangle (2.13,1.19);
\draw [line width=0.4pt,dash pattern=on 1pt off 1pt] (-1.0,0.0) circle (1.0cm);
\draw [line width=0.4pt,dash pattern=on 1pt off 1pt] (1.0,0.0) circle (1.0cm);
\draw (-1.0,1.0)-- (1.0,1.0);
\draw (1.0,-1.0)-- (-1.0,-1.0);
\draw [shift={(-1.0,0.0)}] plot[domain=1.57:4.71,variable=\t]({1.0*1.0*cos(\t r)+-0.0*1.0*sin(\t r)},{0.0*1.0*cos(\t r)+1.0*1.0*sin(\t r)});
\draw [shift={(1.0,0.0)}] plot[domain=-1.57:1.57,variable=\t]({1.0*1.0*cos(\t r)+-0.0*1.0*sin(\t r)},{0.0*1.0*cos(\t r)+1.0*1.0*sin(\t r)});
\begin{scriptsize}
\draw [color=black] (-1.0,0.0)-- ++(-1.0pt,-1.0pt) -- ++(2.0pt,2.0pt) ++(-2.0pt,0) -- ++(2.0pt,-2.0pt);
\draw [color=black] (1.0,0.0)-- ++(-1.0pt,-1.0pt) -- ++(2.0pt,2.0pt) ++(-2.0pt,0) -- ++(2.0pt,-2.0pt);
\draw (0.2,1.1) node {$\hull(\Theta)$};
\end{scriptsize}
\end{tikzpicture}}
\quad
\subfigure{
\begin{tikzpicture}[scale=1.5, line cap=round,line join=round,>=triangle 45,x=1.0cm,y=1.0cm]
\clip(-2.21,-1.35) rectangle (2.13,1.19);
\draw [line width=0.4pt,dash pattern=on 1pt off 1pt] (-1.0,0.0) circle (1.0cm);
\draw [line width=0.4pt,dash pattern=on 1pt off 1pt] (1.0,0.0) circle (1.0cm);
\draw [line width=1.6pt] (-1.54,0.84)-- (1.54,0.84);
\draw [line width=1.6pt] (1.54,-0.84)-- (-1.54,-0.84);
\draw [shift={(-1.0,0.0)},line width=1.6pt]  plot[domain=2.14:4.14,variable=\t]({1.0*1.0*cos(\t r)+-0.0*1.0*sin(\t r)},{0.0*1.0*cos(\t r)+1.0*1.0*sin(\t r)});
\draw [shift={(1.0,0.0)},line width=1.6pt]  plot[domain=-0.99:0.99,variable=\t]({1.0*1.0*cos(\t r)+-0.0*1.0*sin(\t r)},{0.0*1.0*cos(\t r)+1.0*1.0*sin(\t r)});
\begin{scriptsize}
\draw [color=black] (-1.0,0.0)-- ++(-1.0pt,-1.0pt) -- ++(2.0pt,2.0pt) ++(-2.0pt,0) -- ++(2.0pt,-2.0pt);
\draw [color=black] (1.0,0.0)-- ++(-1.0pt,-1.0pt) -- ++(2.0pt,2.0pt) ++(-2.0pt,0) -- ++(2.0pt,-2.0pt);
\draw (0.2,1) node {$M$};
\end{scriptsize}
\end{tikzpicture}}
\caption{The stadium $\hull(\Theta)$ and the mobile $M$.} \label{figure}
\end{figure}

For the sake of the reader we immediately compute the right value of $h$ deducing the value of the Fraenkel 2-asymmetry for the mobile $M$. 
\begin{lemma}\label{lemma.mobile}
The height $h$ in Definition~\ref{mobile} is approximately $0.336$. Therefore, the Fraenkel 2-asymmetry for the mobile $M$ is
\begin{equation}\label{hdjd}
\A_2(M)=\frac{16h}{\sqrt{2\pi}}-2\approx 0.147.
\end{equation}
\end{lemma}
\begin{proof}
From Lemma~\ref{lemma.theta} and Definition~\ref{mobile} it is clear that $\A_2(M)=|M \triangle \Theta|$.
Denoting by $X$ the region delimited by the points $O,Q,S$ (see Figure~\ref{figure.2}), and by $Y$ the set $\{(x,y)\in \Theta\;:\;x<0,\; y>h\}$ yields that $|M\setminus \Theta|=4|X|$ and $|\Theta\setminus M|=4|Y|$, therefore, to compute $\A_2(M)$ it is sufficient to compute the area of $X$ and of $Y$. To this purpose, we introduce the smallest angle $\alpha$ made by the radius $PQ$ and the segment $PR$, where the height $h$ of Definition~\ref{mobile} is linked to $\alpha$ by
\begin{equation}\label{link}
h=\frac{1}{\sqrt{2\pi}}\cos \alpha,
\end{equation}
since $\Theta$ is made by two balls of radius $1/\sqrt{2\pi}$.
Now the area of the rectangle $OPRS$, of the triangle $PQR$ and of the circular sector $OPQ$ are, respectively,
\[
|OPRS|=\frac{h}{\sqrt{2\pi}}, \quad |PQR|=\frac{h}{2}\sqrt{\frac{1}{2\pi}-h^2},\quad  |OPQ|=\frac{1}{4\pi}\Big(\frac{\pi}{2}-\alpha\Big),
\] 
and then
\[
|X|=\frac{h}{\sqrt{2\pi}}-\frac{h}{2}\sqrt{\frac{1}{2\pi}-h^2}-\frac{1}{4\pi}\Big(\frac{\pi}{2}-\alpha\Big), \quad |Y|=\frac{\alpha}{2\pi}-h\sqrt{\frac{1}{2\pi}-h^2}.
\]
The constraint $|M|=|\Theta|=1$ yields \eqref{equivalent} with $E=\Theta$, thus imposing $|X|=|Y|$ and recalling \eqref{link} the height $h$ have to satisfy the following equality:
\[
h\sqrt{\frac{1}{2\pi}-h^2}+\frac{2h}{\sqrt{2\pi}}-\frac{\arccos(\sqrt{2\pi} h)}{2\pi}=\frac{1}{4},
\]
which is solved by $h\approx 0.336$. Moreover, plugging this equation and \eqref{link} into the formula for $|Y|$ we obtain
$\A_2(M)=8|X|=8|Y|$ which provides \eqref{hdjd}.
\end{proof}
\begin{figure}[t]
\centering
\definecolor{uuuuuu}{rgb}{0.267,0.267,0.267}
\begin{tikzpicture}[scale=1.7, line cap=round,line join=round,>=triangle 45,x=1.0cm,y=1.0cm]
\clip(-2.58,-1.25) rectangle (2.47,1.29);
\draw [shift={(-1.0,0.0)},line width=0.0pt,fill=black,fill opacity=0.05] (0,0) -- (57.14:0.26) arc (57.14:90.0:0.26) -- cycle;
\draw [line width=0.4pt] (-1.0,0.0) circle (1.0cm);
\draw [line width=0.4pt] (1.0,0.0) circle (1.0cm);
\draw [line width=0.4pt] (-1.54,0.84)-- (1.54,0.84);
\draw [line width=0.4pt] (1.54,-0.84)-- (-1.54,-0.84);
\draw [shift={(-1.0,0.0)},line width=0.4pt]  plot[domain=2.14:4.14,variable=\t]({1.0*1.0*cos(\t r)+-0.0*1.0*sin(\t r)},{0.0*1.0*cos(\t r)+1.0*1.0*sin(\t r)});
\draw [shift={(1.0,0.0)},line width=0.4pt]  plot[domain=-0.99:0.99,variable=\t]({1.0*1.0*cos(\t r)+-0.0*1.0*sin(\t r)},{0.0*1.0*cos(\t r)+1.0*1.0*sin(\t r)});
\draw (-1.0,0.0)-- (-0.46,0.84);
\draw [line width=0.4pt] (-1,0.84)-- (-1.0,0.0);
\draw [line width=1.2pt] (-0.46,0.84)-- (-1.5425863986500215,0.84);
\draw [shift={(-1.0,0.0)},line width=1.2pt]  plot[domain=0.99:2.14,variable=\t]({1.0*1.0*cos(\t r)+-0.0*1.0*sin(\t r)},{0.0*1.0*cos(\t r)+1.0*1.00*sin(\t r)});
\draw (-0.05,0.86) node[anchor=north west] {$X$};
\draw (-1.3,1.3) node[anchor=north west] {$Y$};
\draw [line width=1.20pt] (-0.46,0.84)-- (0.0,0.84);
\draw [shift={(-1.0,0.0)},line width=1.202pt]  plot[domain=0.0:0.99,variable=\t]({1.0*1.0*cos(\t r)+-0.0*1.0*sin(\t r)},{0.0*1.0*cos(\t r)+1.0*1.0*sin(\t r)});
\draw [line width=1.2pt] (-0.0,-0.0)-- (0.0,0.84);
\draw (1.39,0.40) node[anchor=north west] {$M$};
\draw [line width=0.4pt] (-2.36,0.0)-- (2.33,0.0);
\draw [line width=0.4pt] (0.0,-1.11)-- (0.0,1.2);
\begin{scriptsize}
\draw [color=black] (-1.0,0.0)-- ++(-1.0pt,-1.0pt) -- ++(2.0pt,2.0pt) ++(-2.0pt,0) -- ++(2.0pt,-2.0pt);
\draw[color=black] (-1.1,-0.1) node {$P$};
\draw [color=black] (1.0,0.0)-- ++(-1.0pt,-1.0pt) -- ++(2.0pt,2.0pt) ++(-2.0pt,0) -- ++(2.0pt,-2.0pt);
\draw [fill=uuuuuu] (-0.0,-0.0) circle (0.7pt);
\draw[color=uuuuuu] (0.09,0.07) node {$O$};
\draw [fill=uuuuuu] (-0.4574136013499784,0.84) circle (0.7pt);
\draw[color=uuuuuu] (-0.43,0.95) node {$Q$};
\draw [fill=uuuuuu] (-1,0.84) circle (0.7pt);
\draw[color=uuuuuu] (-1.1,0.92) node {$R$};
\draw[color=black] (-0.9,0.3) node {$\alpha$};
\draw [fill=uuuuuu] (0.0,0.84) circle (0.7pt);
\draw[color=uuuuuu] (0.1,0.93) node {$S$};
\draw[color=uuuuuu] (-1.1,0.5) node {$h$};
\end{scriptsize}
\end{tikzpicture}
\caption{The Fraenkel 2-asymmetry of the mobile $M$.}\label{figure.2}
\end{figure}
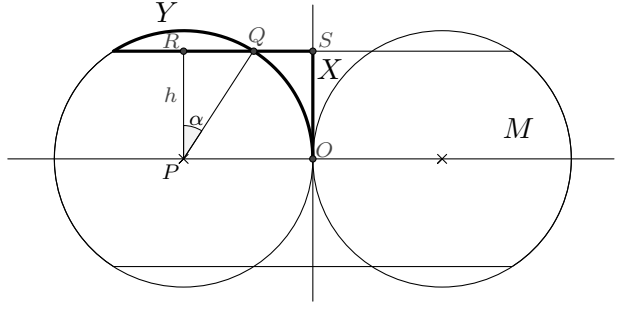

\begin{theorem}[Isoperimetric inequality for the Fraenkel 2-asymmetry]\label{teo.iso}
Among all convex open planar sets with unit area, the \emph{mobile} has the minimum Fraenkel 2-asymmetry, that is for every convex open set $\Omega\subset\R^2$ with $|\Omega|=1$
\begin{equation}\label{isofras}
\A_2(\Omega)\geq \A_2(M),
\end{equation}
and equality holds if and only if $\Omega=M$.
\end{theorem}

\begin{proof}
For every convex and open set $\Omega$, moving and rotating the coordinate system, from Lemma \ref{lemma.theta} we may always assume the minimum in the definition of the Fraenkel 2-asymmetry to be reached by $\Theta$, namely the two balls of equal measure which are tangent in the origin and with centers on the $x$-axis. Therefore, recalling \eqref{equivalent} with $E=\Theta$ and noticing that $\A_2(M)=2{|\Theta\setminus M|}$, in order to prove \eqref{isofras}  it is sufficient to prove the auxiliary inequality
\begin{equation}\label{auxiliary}
|\Theta\setminus \Omega|\geq |\Theta\setminus M|\quad \text{for every convex open set $\Omega\subset\R^2$ with $|\Omega|=1$}.
\end{equation}
We show this inequality in several steps, starting with the existence result, and then listing several necessary conditions that have to be satisfied by an optimal set. 

\smallskip
\emph{i) Existence of an optimal set.}

The existence of an optimal set $\Om^*$ for the minimization problem 
\begin{equation}\label{auxiliary2}
\min \{|\Theta\setminus\Omega| : \; \text{$\Omega\subset \R^2$, $\Omega$ open and convex, $|\Omega|=1$} \},
\end{equation}
is a straightforward consequence of the direct method of the Calculus of Variations. 
Indeed take a minimizing sequence $\{\Om_n\}$ made of open and convex sets of unit area. The diameter $\diam(\Om_n)$ is uniformly bounded, otherwise from the area and the convexity constraints $|\Om_n\cap \Theta|\downarrow 0$, and so $|\Theta\setminus \Om_n|\uparrow 1$, which is strictly greater to the value attained by the mobile $|\Theta\setminus M|\approx 0.073$ (see Lemma~\ref{lemma.mobile}).

Therefore, we can apply Proposition 2.4.3 and 2.4.4 of \cite{BB} to infer that we are minimizing a {continuous} functional $\Om\mapsto |\Theta \setminus \Om|$ over a {compact} class of sets with respect to the uniform convergence.

\smallskip
\emph{ii) An optimal set is included in the stadium $\hull(\Theta)$}. 

Assume, by contradiction, that $\Om^*$ is not contained in the stadium $\hull(\Theta)$, namely the following strict inclusion holds:
\begin{equation}\label{contradiction1}
\Om^*\cap \hull(\Theta)\subset\Om^*.
\end{equation}
Consider the non-decreasing (in the sense of set inclusion) family of dilated sets 
\begin{equation}\label{dilated}
\Omega(r):=\hull(r\Omega^*\cap \Theta), \quad \text{with $r\in[1,\infty)$}.
\end{equation}
For $r=1$, the corresponding set $\Om(1)\subset \Omega^*$ and, in particular, $|\Omega(1)|<|\Omega^*|$. Indeed, the convex set $\Om^*\cap \hull(\Theta)$ contains $\Omega^*\cap \Theta$, thus, by definition of the convex envelope, $\hull(\Omega^*\cap \Theta)\subseteq \Om^*\cap \hull(\Theta)$, which combined with the assumption \eqref{contradiction1} provides the strict inclusion $\Omega(1)\subset\Om^*$.
For $r=\diam \Theta$, $\Omega({\diam \Theta})\supset \hull(\Theta)$, and, in particular, $|\Omega({\diam \Theta})|>|\Omega^*|$. Since the function $|\Omega(r)|$ is continuous over $[1,\diam(\Theta)]$, i.e. for a sequence of real numbers $r_n\to r$, the convex sets ${\Omega(r_n)}$  converge to ${\Omega(r)}$ for the $L^1$ convergence of characteristic functions, there exists ${r_1}>1$ such that the convex set $\Omega(r_1)$ defined by \eqref{dilated} with $r=r_1$ satisfies the volume constraint $|\Omega({r_1})|=|\Omega|=|\Theta|$. Therefore, $\Omega(r_1)$ is an admissible competitor in $\eqref{auxiliary}$ which, by definition \eqref{dilated}, is so that $\Om^*\cap \Theta\subset  \Om({r_1})\cap\Theta$ (note that $r_1$ is strictly greater than $1$). Passing to the complementary sets, this strict inclusion yields that
\[
|\Theta\setminus \Omega^*| > |\Theta\setminus \Omega({r_1})|,
\]
contradicting the optimality of $\Omega^*$. 

\smallskip
\emph{iii)  The closure of the boundary of an optimal set which is not included in $\partial \Theta$ is made by two connected components, each of which is a segment}. 

From the previous step ii), in particular \eqref{dilated}, we know that the boundary of an optimal set which is not included in the two balls $\Theta$ is made by exactly two connected components, each of which is a segment. Then, we can consider the two supporting lines of $\Om^*$ passing through these two segments, denoting by $E$ the region between these two lines containing $\Omega^*$. The region $E$ is a strip, in the case the supporting lines are parallel, otherwise a cone with a vertex outside $\overline{\hull(\Theta)}$ (if these lines would intersect in a point into $\overline{\hull(\Theta)}$, we would have that $|E\cap\hull(\Theta)|<|\Theta|$ and since $\Omega^*\subseteq E\cap \hull(\Theta)$ this would contradict the area constraint). Now assuming, by contradiction, that point iii) does not hold, we would have the strict inequality 
\begin{equation}\label{contradiction2}
|E\cap \hull(\Theta)|>|\Omega^*|.
\end{equation}
Then, we consider the contracted family of convex sets 
\begin{equation}\label{contracted}
\Omega(r):=rE\cap \hull(\Theta), \quad \text{with $r\in(0,1)$},
\end{equation}
where now $rE$ denotes the contraction of the cone $E$ of factor $r$ that keeps fixed its vertex, and we notice that the function $|\Omega(r)|$ is continuous over $(0,1)$, i.e. for a sequence of real numbers $r_n\to r$ in $(0,1)$, the convex sets ${\Omega(r_n)}$ converge to ${\Omega(r)}$ for the $L^1$ convergence of characteristic functions. Moreover, as $r\to1$, from \eqref{contradiction2}, we have that $|\Omega(r)|>|\Omega^*|$, while as $r\to0$, the set $rE$ shrinks to a line, thus in particular, $|\Omega(r)|\to 0$. Therefore, there exists $r_2<1$ such that the convex set $\Omega(r_2)$ defined in \eqref{contracted} satisfies the area constraint $|\Omega(r_2)|=|\Omega|$ and the strict inclusion $\Omega(r_2)\setminus\Theta\subset\Omega^*\setminus\Theta$ gives $|\Omega(r_2)\setminus\Theta|<|\Omega^*\setminus\Theta|$ (since $r_2$ is strictly smaller than $1$). Recalling \eqref{equivalent} with $E=\Theta$ yields
\[
|\Theta\setminus \Omega^*| > |\Theta\setminus \Omega({r_2})|,
\]
contradicting the optimality of $\Omega^*$. 

\smallskip
\emph{iv)  The mobile is the optimal set}.

Combining the steps i) and ii) with iii) we know that an optimal set $\Om^*$ has the following form
\begin{equation}\label{cone}
\Omega^*=E\cap\hull(\Theta),
\end{equation}
where $E$, as before, is a strip or a cone with a vertex outside $\overline{\hull(\Theta)}$ such that $|E\cap\hull(\Theta)|=1$.
Now, the boundary of this set is made by four pieces: two segments and two arc of circles. If for every cone $E_{cone}$ one can find a strip $E_\text{strip}$ which decreases the functional in \eqref{auxiliary2}, then optimizing among all sets of the form $\Omega^*=E_{\text{strip}}\cap\hull(\Theta)$ would give that the symmetric strip is the best one, that is, according to Definition \ref{mobile}, the mobile.
Indeed, assume that the optimal set $\Om^*$ as in \eqref{cone} is generated by a cone $E=E_{\text{cone}}$, and, up to a change of coordinates we may assume that it has on the boundary a segment which belongs to the half-plane $\{y>0\}$ and that is not parallel to the line $\{y=0\}$. We show how it is possible to rotate this segment decreasing the functional in \eqref{auxiliary2}. 
Let us call $l_1$ the supporting line generating this segment, $P$ be the point on this line with $x_{P}=0$, and assume that the following quantity is positive
\[
d:=|\{y>0\}\cap \{x>0\}\cap (E\cap \hull(\Theta))|-|\{y>0\}\cap \{x<0\}\cap (E\cap \hull(\Theta))|> 0.
\] 
Rotating the line $l_1$ around the point $P$, so as to decrease $d$ until $d=0$, yields the line $l_2$ parallel to $\{y=0\}$. Calling $E_2$ the new set obtained from $E_{\text{cone}}$ with $l_2$ in place of $l_1$ we can see that the area has been increased, namely $|E_2\cap \hull(\Theta)|>|E_{\text{cone}}\cap \hull(\Theta)|$, and moreover $|\hull(\Theta)\setminus E_2|<|\hull(\Theta)\setminus E_{\text{cone}}|$. Now, we replace the line $l_2$ with a parallel line $l_3$ and define a new set $E_3$ so that $|E_3\cap \hull(\Theta)|=|E_{\text{cone}}\cap \hull(\Theta)|$. Therefore we have constructed a set $E_3$ which satisfies the area constraint and so that the functional has been decreased $|\hull(\Theta)\setminus E_3|<|\hull(\Theta)\setminus E_{\text{cone}}|$. The same strategy can be adapted if the quantity $d<0$, and for the other segment on the boundary.

At the end, we constructed a better competitor $E_\text{strip}$ and optimizing among all $E_\text{strip}$ will provide that the mobile is the unique minimizer, and the theorem is concluded.
\end{proof}

\begin{remark}\label{rem.general}
If $N>2$ it seems more difficult to find the optimal set in \eqref{min2fras}. Nevertheless, some of the above results can immediately generalized to all dimensions, such as Lemma~\ref{lemma.theta} as well as the steps i), ii) and iii) in the proof of Theorem~\ref{teo.iso}. As a by product, for $N>2$, an optimal set $\Om^*$ in \eqref{min2fras} is so that $\Om^*=E\cap \hull(\Theta)$ where $E$ is an $N$ dimensional cone or strip.
\end{remark}

\section{Convex combinations of the lowest Dirichlet eigenvalues}\label{sec.comb}

\subsection{Basic properties} 

We start discussing some properties of the functional $F_t$ defined in \eqref{convexcomb} that follow from the definition of minimality: we say that $\Om_t$ is a minimizer in \eqref{P} if for every admissible competitor $\Om$
\begin{equation}\label{minimality}
t\la_1(\Om_t)+(1-t)\la_2(\Om_t)\leq t\la_1(\Om)+(1-t)\la_2(\Om),
\end{equation}
and equivalently, rearranging the terms
\begin{equation}\label{minimality2}
\la_1(\Om_t)-\la_1(\Om)+\la_2(\Om)-\la_2(\Omega_t)\leq \frac{1}{t}\Big(\la_2(\Om)-\la_2(\Omega_t)\Big).
\end{equation}

\begin{lemma}
For $s,t\in (0,1)$ with $s<t$, let $\Om_{s}$ and $\Om_{t}$ be minimizers of the functionals $F_{s}$ and $F_{t}$, respectively. The following properties hold:
\begin{itemize}
\item[i)] \emph{the gap non-decreases} $\la_2(\Om_{s})-\la_1(\Om_{s})\leq \la_2(\Om_{t})-\la_1(\Om_{t})$;
\item[ii)] \emph{the first eigenvalue non-increases} $\la_1(\Om_{s})\geq\la_1(\Om_{t})$;
\item[iii)] \emph{the second eigenvalue non-decreases}  $\la_2(\Om_{s})\leq \la_2(\Om_{t})$;
\item[iv)] \emph{a rescaled convex combination increases} ${F_{s}(\Om_{s})}/{(1-s)}<{F_{t}(\Om_{t})}/{(1-t)}$;
\item[v)] \emph{the convex combination decreases} $F_{s}(\Om_{s})>F_{t}(\Om_{t})$.
\end{itemize}
\end{lemma}
\begin{proof}
The minimality \eqref{minimality} of $\Omega_{s}$ for $F_{s}$ with the competitor $\Om=\Om_t$ writes as 
\begin{equation}\label{m1}
s\la_1(\Om_{s})+(1-{s})\la_2(\Om_{s})\leq {s}\la_1(\Om_{t})+(1-{s})\la_2(\Om_{t}),
\end{equation}
while the minimality \eqref{minimality} of $\Omega_{t}$ for $F_{t}$ with the competitor $\Om=\Om_s$, as
\begin{equation}\label{m2}
t\la_1(\Om_{t})+(1-t)\la_2(\Om_{t})\leq t\la_1(\Om_{s})+(1-t)\la_2(\Om_{s}).
\end{equation}
Summing up \eqref{m1} with \eqref{m2} we get to 
\begin{equation}\label{eq1}
(t-s)(\la_1(\Om_{t})-\la_1(\Om_{s}))\leq (t-s)(\la_2(\Om_{t})-\la_2(\Om_{s})),
\end{equation}
and by the assumption $t-s>0$, we immediately obtain the first point i) of this lemma. If $\la_1(\Om_{s})<\la_1(\Om_{t})$ the right-hand side of~\eqref{eq1} would be positive, hence $\la_2(\Om_{s})<\la_2(\Om_{t})$, which would contradict \eqref{m2}.
Then, necessarily, the second point ii) of this lemma holds. Moreover, if $\la_2(\Om_{s})>\la_2(\Om_{t})$, using point ii) we get a contradiction with \eqref{m1}, and so also the third point iii) holds. 
Now, from \eqref{m1} and the fact that ${s/(1-{s})}<t/(1-{t})$, we have
\[
\frac{{s}}{(1-{s})}\la_1(\Om_{s})+\la_2(\Om_{s})\leq \frac{{s}}{(1-{s})}\la_1(\Om_{t})+\la_2(\Om_{t})< 
\frac{{t}}{(1-{t})}\la_1(\Om_{t})+\la_2(\Om_{t}).
\]
Therefore, $F_{t}(\Om)/(1-{t})=t/(1-t)\la_1(\Omega)+\la_2(\Omega)$ and we get point iv). To prove the last point v) we assume, by contradiction, that $F_{s}(\Om_{s})\leq F_{t}(\Om_{t})$. Then recalling \eqref{m2} we have
\[
s\la_1(\Om_{s})+(1-s)\la_2(\Om_{s})\leq t\la_1(\Om_{t})+(1-t)\la_2(\Om_{t})\leq t\la_1(\Om_{s})+(1-t)\la_2(\Om_{s}) ,
\]
which leads to
\[
(t-s)\la_2(\Om_{s})\leq(t-s)\la_1(\Om_{s}).
\]
The assumption $t-s>0$ implies that $\la_1(\Om_{s})=\la_2(\Om_{s})$, contradicting the connectedness of $\Omega_{s}$ proved in \cite{im}.
\end{proof}

\begin{lemma}
Let $s,t\in (0,1)$ with $s<t$. If $X$ is a minimizer of both $F_{s}$ and $F_{t}$, then $X$ is also a minimizer of $F_r$, for every $r\in(s,t)$. 
\end{lemma}
\begin{proof}
We assume, by contradiction, that $X$ is not a minimizer of $F_r$, for some fixed $r\in(s,t)$, and call by $Y$ a minimizer of the corresponding functional $F_{r}$. As in the proof of the previous lemma, using the minimality of $X$ for $F_{s}$ and of $Y$ for $F_{r}$ we arrive to \eqref{eq1} with $t=r$, namely since $s<r$,
\[
(\la_1(Y)-\la_1(X))<(\la_2(Y)-\la_2(X)),
\]
and, similarly, using the minimality of $X$ for $F_{t}$ and again of $Y$ for $F_r$
\[
(\la_1(X)-\la_1(Y))<(\la_2(X)-\la_2(Y)),
\]
where the strict inequalities are a consequence of the assumption $F_r(Y)<F_r(X)$.
Summing up these two inequalities, we reach a contradiction.	
\end{proof}

\subsection{On the non-convexity of the minimizers}

Recently, Henrot and Oudet in \cite{ho} investigated the problem of minimizing the second eigenvalue of the Dirichlet-Laplacian among sets of fixed measure and with an additional \emph{convexity} contraint. 
Finding an explicit minimizer in this class seems a very difficult problem: a possible candidate to be the optimum is the stadium (i.e., the convex hull of two tangent balls), but this conjecture was refuted in \cite{ho}. Indeed any set which contains on the boundary some pieces of balls can not be a minimizer.
Nevertheless, in \cite{ho} it was proved the existence of a convex minimizer $\ho$ so that, for every open and \emph{convex} set $\Om\subset\R^N$ with unit area,
\begin{equation}\label{HO}
\la_2(\Om)\geq\la_2(\ho),
\end{equation} 
(cf. \eqref{HO} with the Krahn-Szeg\"o inequality, where no-convexity constraint is required). Notice that, since $\ho$ has no pieces of balls on its boundary, in particular $\ho\neq B$ and
\begin{equation*}
\omega_N^{2/N} j_{N/2}^2=\la_2(B)>\la_2(\ho).
\end{equation*}
In two dimensions, Oudet in \cite{o} and, more rencently, Antunes and Henrot in \cite{ah}, made some numerical computations, showing the shape of the optimal set $\ho$ and highlighting that $\ho$ is very close to the stadium, both from a geometrical and a numerical point of view; in particular
\begin{equation}\label{la2values}
\la_2(B_-\cup B_+)=2\pi j_{0}^2\approx36.336,\quad\la_2(\ho)\approx 37.987,\quad\la_2(\Om_{\text{stadium}})\approx 38.002,
\end{equation}
where $\Om_{\text{stadium}}$ is the stadium with $|\Om_\text{stadium}|=1$, i.e., a contracted version of the set $\hull(\Theta)$. 

We are now in position to prove Theorem \ref{teo.1}.

\begin{proof}[Proof of Theorem~\ref{teo.1}]
From the Krahn-Szeg\"o inequality \eqref{KS} and the connectedness of $\Om_t$ it follows that
\begin{equation}\label{positive3}
\la_2(B_-\cup B_+)<\la_2(\Om_t),
\end{equation}
which plugged into \eqref{minimality} with $\Om=B_-\cup B_+$ yields
\begin{equation}\label{positive2}
\la_1(\Om_t)<\la_1(B_-\cup B_+).
\end{equation}
Taking $\Om=B_-\cup B_+$ also in \eqref{minimality2} and dividing therein by the negative quantity $\la_2(B_-\cup B_+)-\la_2(\Omega_t)$  (recall \eqref{positive3}) we get to
\begin{equation}\label{const2}
\frac{\la_1(B_-\cup B_+)-\la_1(\Om_t)}{\la_2(\Om_t)-\la_2(B_-\cup B_+)}+1\geq \frac{1}{t}.
\end{equation}
From \eqref{positive3} and \eqref{positive2}, the ratio on the left-hand side of this inequality turns out to be a positive number; therefore, we can use the Faber-Krahn inequality \eqref{FK} to estimate $\la_1(\Om_t)$ at the numerator of this ratio. Moreover, if $\Om_t$ would be a \emph{convex} set, we could also use \eqref{HO} to estimate $\la_2(\Om_t)$ at the denominator of this ratio, obtaining the following uniform bound on $t$:
\begin{equation}\label{const}
t\geq \frac{1}{\frac{\la_1(B_-\cup B_+)-\la_1(B)}{\la_2(\ho)-\la_2(B_-\cup B_+)}+1}.
\end{equation}
Calling $T$ the quantity on the right-hand side of this inequality, the Krahn-Szeg\" o inequality for convex sets gives $\la_2(\ho)-\la_2(B_-\cup B_+)>0$, thus $T>0$. Therefore, if $t<T$, $\Omega_t$ can not be convex.
\end{proof}

The proof of Theorem~\ref{teo.1} is constructive and reveals an explicit expression for the threshold $T$ in terms of the eigenvalues of the Dirichlet-Laplacian. 
\begin{corollary}\label{cor3.3}
The threshold $T$ in Theorem~\ref{teo.1} has the following expression:
\begin{equation}\label{explicit}
T=1-\frac{(2^{2/N}-1)\la_1(B)}{\la_2(\ho)-\la_1(B)},
\end{equation}
where $\ho$ is a minimizer in \eqref{HO}.  In two dimensions, it turns out that 
\begin{equation}\label{bounds2}
T\geq 1-\frac{1}{1+2\betaks\mathcal{A}_2(M)^{9/2}}\approx 1.192\cdot 10^{-14},
\end{equation}
where the constants $\betaks$ and $\A_2(M)$ are as in~\eqref{constantKS} and~\eqref{isofras} respectively.
\end{corollary}
\begin{proof}
Define as $T$ the quantity on the right-hand side of \eqref{const}. Noticing that $\la_1(B_-\cup B_+)=\la_2(B_-\cup B_+)=2^{2/N}\la_1(B)$ and simplifying, we reach \eqref{explicit}. Moreover, plugging, the quantitative Krahn-Szeg\"o inequality \eqref{QKS} into \eqref{explicit}, yields a lower bound on $T$, which is independent on the eigenvalues of the Dirichlet-Laplacian
\[
T\geq 1-\frac{(2^{2/N}-1)}{(2^{2/N}-1)+2^{2/N} C_{KS} \mathcal{A}_2(\ho)^{2(N+1)}}.
\]
In two dimensions, on the other hand, one can use~\eqref{QKSbha},~\eqref{constantKS} and Theorem~\ref{teo.iso} together with~\eqref{hdjd} to get the explicit value in \eqref{bounds2}.
\end{proof}

\begin{remark}
The explicit value for the lower bound to the threshold $T$ is not very accurate, mostly due to the fact that the constant $\betaks$ is not the optimal one, but we believe it is important to show that a numerical value can actually be provided.
Moreover, if $N=2$, plugging the numerical computation of $\la_2(\ho)$ recalled in \eqref{la2values} into \eqref{explicit} and using $\la_1(B)=\pi j_{0}^2\approx18.168$, reveals a numerical approximation for the threshold defined by \eqref{explicit}:
\[
T\approx 0.083.
\]
\end{remark}

\subsection{The ball never minimizes the convex combination}
For the proof of Theorem \ref{teo.2} we need the following result. 

\begin{prop}\label{la2simple}
For a fixed $t\in(0,1)$, let $\Om_t$ be a minimizer of problem~\eqref{P}.
If the boundary of $\Om_t$ is of class $C^2$ then $\la_2(\Omega_t)$ is simple, namely $\la_1(\Om_t)<\la_2(\Om_t)<\la_3(\Om_t)$.
Moreover, on the boundary of $\Omega_t$, the following optimality condition holds:
\begin{equation}\label{optim12}
t\left|\nabla u_1(x)\right|^2+(1-t)\left|\nabla u_2(x)\right|^2=\frac{2F_t(\Omega_t)}{N}, \qquad x\in\partial \Om_t.
\end{equation}
\end{prop}
\begin{proof}
The fact that $\la_1(\Om_t)<\la_2(\Om_t)$ holds since $\Om_t$ is connected, as it was proved in~\cite{im}.
To prove that $\la_2(\Om_t)$ is simple we proceed as in~\cite[Lemma 2.1]{bbh} (see also \cite[Lemma 2.5.9]{H}) assuming, by contradiction, that $\la_2(\Omega_t)$ is a multiple eigenvalue, in the sense that $\la_2(\Om_t)=\dots=\la_k(\Om_t)$ for some $k\geq3$ (see \cite{H} for a precise definition of the higher eigenvalues). 
We use the results of derivability of multiple eigenvalues, see for instance~\cite[Theorem~2.5.8]{H}.
We deform the domain $\Om_t$ with a regular vector field $\phi_\epsilon(x)= x+\epsilon V(x)$ and we denote with $\{u_i\}_{2\leq i\leq k}$ the orthonormal (for the $L^2$ scalar product) family of eigenfunctions associated to $\la_2(\Om_t)$.
The directional derivative of the map $\epsilon\mapsto \la_2(\phi_{\epsilon}(\Omega_t))$ at $\epsilon=0$ is one of the eigenvalues of the $(k-1)\times (k-1)$ matrix
\begin{equation}\label{derivative1}
A=\left(-\int_{\partial \Om_t}\frac{\partial u_i}{\partial \nu}\frac{\partial u_j}{\partial \nu}\, V. \nu \,d\sigma \right)_{2\leq i,j\leq k},
\end{equation}
where $V. \nu $ denotes the normal displacement of the boundary induced by the vector field $V$. Moreover, as observed before, the first eigenvalue is simple at $\Om_t$, then it is differentiable and the derivative is a linear form in $V.\nu$ supported on $\partial\Om_t$ (see e.g. \cite[Theorem 2.5.1]{H})
\begin{equation}\label{derivative2}
d\la_1(\Omega_t; V)=-\int_{\partial\Om_t}\left(\frac{\partial u_1}{\partial \nu}\right)^2V.\nu \, d\sigma,
\end{equation}
while the derivative of the volume is given by
\begin{equation}\label{derivative3}
d\text{Vol}(\Omega_t; V)=\int_{\partial\Om_t}V.\nu \, d\sigma.
\end{equation}

Now, since $\Om_t$ is a minimizer in \eqref{P}, then it is also a minimizer for the Lagrangian
\begin{equation}\label{lagrangian}
\mathcal L(\Om)=t\la_1(\Om)+(1-t)\la_2(\Om)+\mu|\Om|, \quad \text{with} \quad \mu:=\frac{2F_t(\Om_t)}{N}.
\end{equation}
Indeed, for such a $\mu$, the function $f(s)=\mathcal L(s\Om)$ of the real variable $s>0$ achieves its minimum in $s=1$, and since $\Om_t$ minimizes $F_t$ this implies that $\mathcal L(\Om_t)\leq \mathcal L(\Om)$ for every bounded open set $\Om\subset\R^N$ (see for instance \cite[Remark 3.6]{DPV}).
Then, we can differentiate the lagrangian $\mathcal{L}$ without taking care of the volume constraint, and from \eqref{derivative1},\eqref{derivative2} and \eqref{derivative3} the derivative of \eqref{lagrangian} is the smallest eigenvalue of the matrix
\[
A_t=td\la_1(\Omega_t;V)\mathop{Id}+(1-t)A+\mu d\text{Vol}(\Omega;V)\mathop{Id},
\]
where $\mathop{Id}$ is the $(k-1)\times (k-1)$ identity matrix.
Therefore, in order to get a contradiction with the optimality of $\Om_t$, it is enough to prove that there is always a vector field $V$ such that the matrix $A_t$ has a negative eigenvalue. To this purpose let us consider two points $P$ and $Q$ on $\partial \Om_t$ and two small neighborhoods $\gamma_{P}$ and $\gamma_{Q}$ of these two points of same length, say $2\delta$. Let us choose any regular function $\varphi$ defined on $(-\delta, \delta)$, vanishing on the boundary of the interval, and a deformation field $V$ such that:
\[
V. \nu=\varphi \;\mbox{on }\gamma_{P},\quad V. \nu=-\varphi\; \mbox{on }\gamma_Q,\quad V. \nu=0\;\mbox{elsewhere on }\partial \Om_t.
\]
With this choice of the field $V$, we have that the matrix $A$ in \eqref{derivative1} becomes the difference of two matrices
\[
A=A(P)-A(Q):=\left(-\int_{\gamma_P}{\frac{\partial u_i}{\partial \nu}\,\frac{\partial u_j}{\partial \nu}\varphi\,d\sigma}\right)_{2\leq i,j\leq k}+\left(\int_{\gamma_Q}{\frac{\partial u_i}{\partial \nu}\,\frac{\partial u_j}{\partial \nu}\varphi\,d\sigma}\right)_{2\leq i,j\leq k}
\]
and so the matrix $A_t$ becomes the difference of two matrices $A_t(P)- A_t(Q)$, where 
\[
A_t(P)=-t\int_{\gamma_P}{\left(\frac{\partial u_1}{\partial \nu}\right)^2\varphi\,d\sigma}\,Id+(1-t)A(P)+\mu \int_{\gamma_P}{\varphi\,d\sigma} \,Id,
\]
and $A_t(Q)$ defined analogously. The only case in which we cannot choose two points $P,Q$ and a function $\varphi $ such that the matrix has a negative eigenvalue is when $A_t(P)=A_t(Q)$, and this implies that, for every $2\leq i,j\leq k$
\begin{equation*}
\int_{\gamma_P}\frac{\partial u_i}{\partial \nu}\,\frac{\partial u_j}{\partial \nu}\,\varphi\,d\sigma=\int_{\gamma_Q}\frac{\partial u_i}{\partial \nu}\,\frac{\partial u_j}{\partial \nu}\,\varphi\,d\sigma,   
\end{equation*} 
and 	
\begin{equation*}
\int_{\gamma_P}\left[\left(\frac{\partial u_i}{\partial \nu}\right)^2-\left(\frac{\partial u_j}{\partial \nu}\right)^2\right]\,\varphi\,d\sigma=\int_{\gamma_Q}\left[\left(\frac{\partial u_i}{\partial \nu}\right)^2-\left(\frac{\partial u_j}{\partial \nu}\right)^2\right]\,\varphi\,d\sigma,
\end{equation*}
for every $P,Q\in\partial \Om_t$ and for every regular function $\varphi$.
It means, in particular, that the product $({\partial u_2}/{\partial \nu})({\partial u_3}/{\partial \nu})$ and the difference $({\partial u_2}/{\partial \nu})^2-({\partial u_3}/{\partial \nu})^2$ should be constant on $\partial \Om_t$,
and then $({\partial u_2}/{\partial \nu})^2$ is constant on $\partial \Om_t$. Applying the classical Serrin Theorem~\cite{serrin} to $u_2$ (or possibly to $-u_2$), in the slightly more general version~\cite[Theorem 18]{fg}, implies that $\Om_t$ must be a ball, which is a contradiction since $\partial u_2/\partial \nu$ cannot be constant on the boundary of a ball (see for example~\cite[Section 1.2.5]{H}). Then, also $\la_2(\Om_t)$ has to be simple.
	
At $\Om_t$ the derivative of the Lagrangian \eqref{lagrangian} has to be zero: then \eqref{derivative2}, \eqref{derivative3}, and the corresponding formula for the derivative of the second eigenvalue $\la_2$ \eqref{derivative1} with $k=2$ (since it is simple) yields that
\begin{equation*}
t\int_{\partial\Om_t}\left(\frac{\partial u_1}{\partial \nu}\right)^2V.\nu \, d\sigma+(1-t)\int_{\partial\Om_t}\left(\frac{\partial u_2}{\partial \nu}\right)^2V.\nu \, d\sigma=\mu\int_{\partial\Om_t}V.\nu \, d\sigma
\end{equation*}
for every regular vector field $V$.  
From this, recalling that in the Dirichlet case $|\nabla u_i|^2=(\partial u_i/\partial \nu)^2$ with $i=1,2$, we obtain \eqref{optim12}.
\end{proof}

\begin{remark}
The simplicity of $\la_2(\Om_t)$ can also be proved under other regularity assumptions on $\Om_t$.
The weakest assumption that is needed in order to differentiate the domain is that the boundary $\partial \Omega_t$ contains a part
\begin{equation}\label{assu}
\emph{\text{ $\Gamma$ which is nonempty, relatively open in $\partial \Omega_t$, connected and of class $C^{1,1}$}.}
\end{equation}
Then, it is enough to repeat the same proof by taking care that the vector field $V$ constructed therein is chosen with support contained in $\Gamma$, and one obtains an overdetermined condition only along that part $\Gamma$ of the boundary. If a Serrin principle is also available for a \emph{partially overdetermined} condition we are done. Indeed Proposition~\ref{la2simple} holds under different regularity assumptions on $\Om_t$, for instance:
\begin{itemize}
\item[-] if $\Omega_t$ is \emph{convex} with $\Gamma$ as in \eqref{assu} (using \cite[Theorem 7]{fg}); 
\item[-] if $\partial \Omega_t$ is connected with $\Gamma$ as in \eqref{assu} and analytically continuable (using \cite[Theorem 1]{fg}). 
\end{itemize}
\end{remark}

\begin{proof}[Proof of Theorem~\ref{teo.2}]
The proof is a straightforward consequence of Proposition~\ref{la2simple}: in every dimension, the second eigenvalue $\la_2(B)$ is not simple, therefore the ball $B$ can not be a minimizer for any $t\in(0,1)$.
\end{proof}

\begin{remark}\label{nearly}
In two dimensions, the fact that balls are never minimizers was implicitly contained in the work ~\cite{KW}. For an arbitrary $\eps>0$ small enough, in \cite{KW} the authors constructed a nearly spherical competitor $B_\eps$, with $|B_\eps|=1$, 
such that 
\[
\la_1(B_\eps)\leq \la_1(B)+d_1\eps^2,\quad\mbox{while }\quad \la_2(B_\eps)\leq \la_2(B)-d_2\eps,
\]
for some positive constants $d_1,d_2$.
Therefore, for every $t\in(0,1)$, it is possible to find $\eps>0$ so small so that 
\[
t\la_1(B_\eps)+(1-t)\la_2(B_\eps)<t\la_1(B)+(1-t)\la_2(B).
\]
\end{remark}

\section{Remarks on the attainable set of the lowest Dirichlet eigenvalues}\label{sec.atta}

We start listing the most important properties that are known on the attainable set $\E$ defined in \eqref{attainable} (see Figure~\ref{figure.3}):
\begin{enumerate}
	\item lies above the bisector $y=x$ (since by definition $\la_2(\Omega)\geq \la_1(\Omega)$ for every $\Omega\subset \R^N$).
	\item lies on the right of the line $x=\la_1(B)$ (for the Faber-Krahn inequality \eqref{FK}).
	\item lies above the line $y=\la_2(B_-\cup B_+)$ (for the Krahn-Szeg\"o inequality  \eqref{KS}).
	\item lies below the line $y=\frac{\la_2(B)}{\la_1(B)}x$ (for the Ashbaugh-Benguria inequality \cite{AB}).
	\item is conical with respect to the origin.
\end{enumerate}

The numerical picture provided by Keller and Wolff suggests the evidence that the attainable set $\E$ is convex, but this is a long-standing conjecture which is still unsolved.

\begin{conjecture}\label{longstanding}
The attainable set $\E$ is convex.
\end{conjecture}

The most important result in the direction of this conjecture was proposed by Bucur, Buttazzo and Figuereido in~\cite{bbf}. These authors proved that the attainable set \eqref{attainable}, constructed through quasi-open set instead of open set, is convex in the vertical and in the horizontal direction and, as a consequence, that it is closed. Nevertheless the vertical and horizontal convexity \emph{do not} imply convexity (think, for example to an L-shaped set).

From the properties of the set $\E$ listed above it is clear that the unique unknown part of the boundary of $\E$ is the curve $\C$ connecting the points $P_B=(\la_1(B),\la_2(B))$ and $P_{B_-\cup B_+}=(\la_1(B_-\cup B_+),\la_2(B_-\cup B_+))$. The convexity of $\E$ is then guaranteed as soon as $\C$ can be parametrized by a convex function.  For this reason it is important to have more informations on the curve $\C$. In two dimensions, Keller and Wolf in~\cite{KW} showed that the tangent of $\mathcal C$ at the point $P_B$ corresponding to a ball $B$ is vertical. They constructed a nearly spherical perturbation of $B$, as recalled in Remark~\ref{nearly}, and then they computed the slope of the tangent to $\C$ as $\epsilon\to0$. Moreover, in all dimensions, Brasco, Nitsch and Pratelli showed that the tangent of $\mathcal C$ at the point $P_{B_-\cup B_+}$ corresponding to two balls $B_-\cup B_+$ is horizontal. In this case the limit as $\epsilon\to 0$ was computed by overlapping the two balls $B_-$ and $B_+$ of a quantity measured in terms of the parameter $\epsilon$.
In the following proposition we recover these limits relying on the minimality condition of the minimizers of convex combinations \eqref{minimality} without any explicit construction. Notice that the strategy that we adopt holds in all dimensions.

\begin{figure}
\begin{tikzpicture}[domain=0:10, scale=0.6] 
  
  \draw[->] (-0.2,0) -- (10,0) node[below] {$x$};
  \draw[->] (0,-0.3) -- (0,11) node[left] {$y$};
  
   \fill[black] (7,9.5) circle (0pt) node {\Large $\E$};
   \draw[thick, domain=36.336/6:10]   plot (\x,{(\x)}) node[right] {\small $y=x$};
   \draw[very thin, dashed,domain=0:36.336/6]   plot (\x,{(\x)});
   \draw[thick, domain=18.168/6:4]   plot (\x,{(\x)*2.54}) node[left] {\small $y=\frac{\la_2(B)}{\la_1(B)}x$};
   \draw[very thin, dashed,domain=0:3.03]   plot (\x,{(\x)*2.539});
   \fill[black] (18.168/6,18.168*2.54/6) circle (2.5pt) node[above left] {\small $P_{B}$};;
   \fill[black] (36.336/6,36.336/6) circle (2.5pt) node[right] {\small $P_{B_-\cup B_+}$};
  
   \draw[thick] (18.168/6,18.168*2.54/6) arc (180:270:3.028cm and 1.63512cm);
   \draw[domain=-0.2:7]   plot (\x,{-0.7*(\x)+9.23}) node[below] {\small $y=\frac{t}{t-1}x+\frac{F_t(\Om_t)}{1-t}$}; 
    \fill[black] (3.66,6.675) circle (2.5pt) node[below] {\small $P_{\Om_t}$}; 
    \fill[black] (5,6.5) circle (0pt) node {$\C$};

\end{tikzpicture}
\caption{The attainable set for the lowest eigenvalues.}\label{figure.3}
\end{figure}
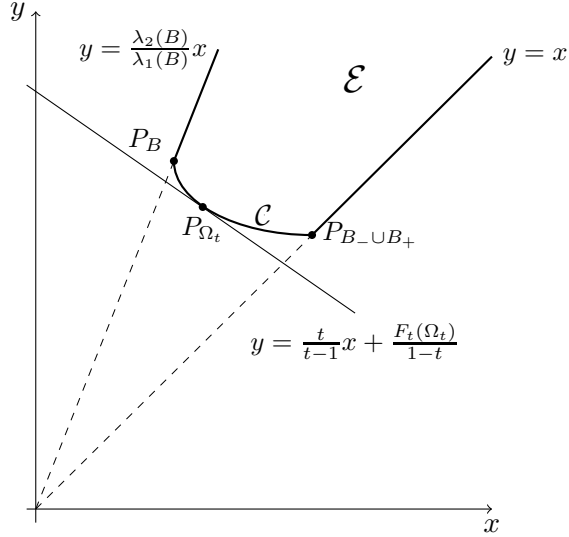

\begin{theorem}\label{teo.tangent}
For every dimension $N\geq 2$ and $t\in(0,1)$, let $\Om_t$ be a minimizer of problem~\eqref{P}. Then we have:
\begin{itemize}
\item[i)] the tangent of $\mathcal C$ at the point $P_B$ corresponding to one ball is \emph{vertical}, namely
\begin{equation}\label{limit1}
\lim_{t\rightarrow 1}{\frac{\la_2(\Om_t)-\la_2(B)}{\la_1(\Om_t)-\la_1(B)}}=-\infty
\end{equation}
	
\item[ii)] the tangent of $\mathcal C$ at the point $P_{B_-\cup B_+}$ corresponding to two identical balls is \emph{horizontal}, namely\begin{equation}\label{limit2}
\lim_{t\rightarrow 0}{\frac{\la_2(\Om_t)-\la_2(B_-\cup B_+)}{\la_1(\Om_t)-\la_1(B_-\cup B_+)}}=0.
\end{equation}
\end{itemize}
	
Moreover, the following limits holds
	\begin{equation}\label{limits3}
	\lim_{t\rightarrow 0}{\la_2(\Om_t)}=\la_2(B_-\cup B_+) \qquad \text{and}\qquad \lim_{t\rightarrow 1}{\la_1(\Om_t)}=\la_1(B).
	\end{equation}
\end{theorem}
\begin{proof}
From the Faber-Krahn inequality \eqref{FK} and Theorem~\ref{teo.2} we find that 
\begin{equation}\label{positive0}
\la_1(B)<\la_1(\Om_t),
\end{equation}
which plugged into \eqref{minimality} with $\Om=B$ yields
\begin{equation}\label{positive}
\la_2(\Om_t)<\la_2(B).
\end{equation}
Taking $\Om=B$ in \eqref{minimality2} and dividing therein by $\la_2(B)-\la_2(\Omega_t)$  (which from \eqref{positive} is a strictly positive value) yields
\[
\frac{\la_1(\Om_t)-\la_1(B)}{\la_2(B)-\la_2(\Om_t)}+1\leq \frac{1}{t}. 
\]
From \eqref{positive0} and \eqref{positive} one can see that the ratio on the left-hand side of this inequality is a positive number, therefore,  letting $t\uparrow 1$, necessarily, it holds the limit in \eqref{limit1}. Moreover, repeating the computations made in the proof of Theorem \ref{teo.1} and letting $t\downarrow 0$ in \eqref{const2}, it follows the limit in \eqref{limit2}.

Finally, the limits in \eqref{limits3} are a consequence of \eqref{limit1}, \eqref{limit2} and of the boundedness of the denominator in \eqref{limit2} (because of \eqref{positive2}) and of the numerator in \eqref{limit1} (because of \eqref{positive}).
\end{proof}

\begin{remark}
The strategy of looking at the boundary of the attainable set through convex combinations can be applied to other attainable sets with different constraints, for instance to the attainable set with a perimeter constraint \cite{AH2}
\[
\E_{\text{p}}:=\left\{(x,y)\in\R^2\;:\;x=\la_1(\Om),\; y=\la_2(\Om),\;\Om\subset\R^N, \;\mbox{$\Om$ open},\;\mathcal{P}(\Om)=1\right\},
\]
or to the attainable set with a convexity constraint \cite{ah}
\[
\E_{\text{c}}:=\left\{(x,y)\in\R^2:\text{$x=\la_1(\Om)$, $y=\la_2(\Om)$, $\Om\subset\R^N$, $\Om$ open and convex, $|\Om|=1$}\right\}.
\]
In particular, as in Theorem~\ref{teo.tangent}, it is possible to show that the tangent of $\partial \E_{\text{p}}$ (or of $\partial \E_{\text{c}}$) at the point $P_B$ is \emph{vertical}.
\end{remark}

We finish this discussion formulating an \emph{isospectral} conjecture on the minimizers of problem~\eqref{P}, which could be used to prove the convexity of the attainable set $\E$.

\begin{conjecture}\label{conj}
Let $t\in(0,1)$ and assume $X,Y\subset \R^N$ to be minimizers of problem \eqref{P} with $F_t(X)=F_t(Y)$. Then, the lowest eigenvalues of $X$ and $Y$ coincide, namely
\[
\la_1(X)=\la_1(Y) \qquad \text{and}\qquad \la_2(X)=\la_2(Y).
\]
\end{conjecture}	

\begin{prop}
The validity of Conjecture~\ref{conj} implies that Conjecture~\ref{longstanding} holds true.
\end{prop}
\begin{proof}
If $\E$ is not convex, then we can find two points $P_X,P_Y\in\C$, corresponding, respectively, to $X,Y$, and a straight line $l$ passing through these points such that the curve $\C$ lies above $l$.
Therefore, it is clear that $l$ will be of the form $t x+(1- t)y=a$ for some fixed $t\in (0,1)$ and a real number $a$.
Hence the sets $X,Y$ are minimizers in \eqref{P} for such a $t$, but $\la_1(X)\not=\la_1(Y)$ and $\la_2(X)\not=\la_2(Y)$, a contradiction with Conjecture~\ref{conj}. 
\end{proof}

\section{Appendix: explicit constants in quantitative inequalitites}\label{appendix}

In the following we compute explicit constants in some quantitative inequalities. 
We focus in particular, on the quantitative Krahn-Szeg\"o inequality \eqref{QKS},
having in mind its application in Corollary~\ref{cor3.3}. We start with a brief overview on the most important quantitative inequalities, without pretending of being exhaustive.

The \emph{quantitative isoperimetric inequality for the De Giorgi perimeter} $\mathcal P$ in the sharp version was proved by Fusco, Maggi and Pratelli in~\cite{fmp}: there exist a constant $C_{I}>0$ (depending on the dimension) such that for all open sets $\Om\subset\R^N$ of finite measure $|\Om|=|B|$,
\begin{equation}\label{QI}
\frac{\mathcal P(\Om)}{\mathcal P(B)}-1\geq C_{ I}\A(\Om)^2,
\end{equation} 
where $\A$ is as in \eqref{2fras1}.
In this setting, a quantitative inequality is \emph{sharp} when it has the least exponent on the Fraenkel asymmetry, and a constant is \emph{optimal} when it has the largest possible value. Moreover, the quantity estimated from below through the asymmetry is the \emph{deficit} of the functional.

The same authors later proved a \emph{quantitative Faber-Krahn inequality}, with a proof based on~\eqref{QI}: there exists a constant $C_{FK}>0$ (depending on the dimension) such that for all open sets $\Om\subset\R^N$ of finite measure $|\Om|=|B|$,
\begin{equation}\label{QFK}
\frac{\la_1(\Om)}{\la_1(B)}-1\geq C_{FK}\A(\Om)^4.
\end{equation} 
Recently Brasco, De Philippis and Velichkov proved a sharp quantitative Faber-Krahn inequality~\cite{bdpv} with the exponent $2$ instead of $4$ in \eqref{QFK}. Unfortunately, their proof relies on the so called \emph{selection principle}, which does not allow to get an explicit constant.

On the other hand a \emph{quantitative Krahn-Szeg\"o inequality}~\eqref{QKS} was proved by Brasco and Pratelli in~\cite{bp}, based on~\eqref{QFK}. Their proof has the good feature of being easily adaptable once a better quantitative Faber-Krahn inequality is available. 

Finding optimal constants in quantitative inequalities is a quite difficult task.
Indeed, up to our knowledge, even the optimal constant for the quantitative isoperimetric inequality \eqref{QI} in the plane is not explicitly known (see, for instance~\cite{cl} and \cite{BCH}). The optimal constant for the quantitative isoperimetric inequality \eqref{QI} is explicitly known only among \emph{convex} sets of the plane, as was proved in \cite{afn}. Moreover, it is also difficult to find explicit non-optimal constants, since often the techniques used in the proofs do not allow to identify constants. Nevertheless, a non-optimal constant for the quantitative isoperimetric inequality~\eqref{QI} was obtained in~\cite{fimp}.
The constant for the quantitative isoperimetric inequality is important when looking for explicit constants in the quantitative Faber-Krahn and Krahn-Szeg\"o inequalities, since both proofs of~\eqref{QFK} and \eqref{QKS} rely on the quantitative isoperimetric inequality \eqref{QI}. With the path just outlined, it is then possible to find an explicit constant for~\eqref{QKS} for all $N\geq2$, although the computations are rather unpleasant and the constants are far from being optimal.

However, in two dimensions, we can follow a different strategy.  In order to get an explicit constant for the quantitative Krahn-Szeg\"o inequality, we use the quantitative Faber-Krahn inequality provide by Bhattacharya in~\cite{bha}. The key point of this inequality is that it  does not rely on a quantitative isoperimetric inequality. Therefore, it can be obtained a better explicit constant $\betafk$ for the quantitative Faber-Krahn inequality, which plugged into the proof of the quantitative Krahn-Szeg\"o inequality \eqref{QKS} allows to improve the constant $\betaks$ as well. To prove this result we go through the papers \cite{bha} and \cite{bp} step-by-step and we highlight the points
where an explicit constant is needed, using, in particular, the same notation of the paper involved.

\subsection{An explicit constant in the quantitative Faber-Krahn inequality}
The quantitative Faber-Krahn inequality in the form proved by Bhattacharya~\cite{bha} reads as: there exists a constant $\betafk>0$ such that for all open sets $\Om\subset \R^2$ of finite measure $|\Om|=|B|$,
\begin{equation}\label{QFKbha}
fk(\Om):=\frac{\la_1(\Om)}{\la_1(B)}-1\geq \betafk\A(\Om)^3.
\end{equation}  
We show that the constant $\betafk$ in~\eqref{QFKbha} can be chosen as follows:
\begin{equation}\label{CQFK}
	\betafk=\frac{1}{10^{5}\cdot 2^3\cdot j_0^2}\approx 2.161\cdot 10^{-7}.
\end{equation}
We use the same notation of~\cite{bha}, noticing that the asymmetry $\alpha$ used there differs from \eqref{2fras1} for a factor 2. Indeed using \eqref{equivalent} for an open set $\Om\subset\R^2$ of unit measure and $E=B$ we have that\[
\alpha(\Om)=\frac{\A(\Om)}{2}.
\]
It is enough to study the case when $fk(\Om)\leq 1$ since once
\begin{equation}\label{eq.appendix0}
fk(\Om)\geq K\alpha(\Om)^3,\qquad \mbox{if }fk(\Om)\leq 1,
\end{equation}
is established, then immediately \eqref{QFKbha} holds true with $\betafk=\min\{K,1\}/2^3$ (recall that $\alpha(\Om)\leq 1$ by definition).
Therefore, let us assume the deficit $fk(\Om)\leq 1$. The first point of~\cite{bha} is Lemma~3.2: it is shown a bound on the ${L^\infty}$-norm of the first eigenfunction $u_1$, that is 
\begin{equation}\label{upperu1}
\|u_1\|_{L^\infty}\leq \frac{\la_1(\Om)}{2\pi}\leq \frac{\la_1(B)}{\pi}=j_0^2,
\end{equation}
where the second inequality holds thanks to the assumption $fk(\Om)\leq 1$, while the last equality comes from the explicit value of $\la_1(B)$ recalled in \eqref{FK}.

Now we can directly pass to analyze Section 4 of~\cite{bha}, where the proof of the main result is carried out.
We recall the constants that will be used:
\begin{equation}\label{khhhj}
p=2,\qquad  k=\frac{1}{625}=\frac{1}{5^4}.
\end{equation}
Only Case 1 is of our interest, since Case 2 deals with $p<2$.
In subcase $(i)$ one gets 
\begin{equation}\label{eq.appendix1}
fk(\Om)\geq \frac{1}{10^5\,j_0^2}\,\alpha(\Om)^3,
\end{equation}
where we used \eqref{khhhj} and $M$ therein defined as an upper bound for $\|u_1\|_{L^\infty}$ and so, according to \eqref{upperu1}, can be taken as $j_0^2$. 
On the other hand, subcase $(ii)$ gives
\begin{equation}\label{eq.appendix2}
fk(\Om)\geq \frac{61}{200}\alpha(\Om)\geq  \frac{61}{200}\alpha(\Om)^3,
\end{equation}
where the second inequality holds since $\alpha(\Om)\leq 1$. Then, combining \eqref{eq.appendix1} with \eqref{eq.appendix2} we get \eqref{eq.appendix0} with $K={1/(10^5\cdot j_0^2)}<1$,
and therefore, from the previous observation $\betafk=K/2^3$, providing \eqref{CQFK}.

\subsection{An explicit constant in the quantitative Krahn-Szeg\"o inequality}

We can now derive an explicit constant $\betaks>0$ for~\eqref{QKSbha}. 
We go through the proof of the quantitative Krahn-Szeg\"o inequality of~\cite{bp} and give the explicit value of the constants that are needed in the proofs, using the quantitative Faber-Krahn~\eqref{QFKbha} instead of~\eqref{QFK}, as it was originally done in that paper.
First of all we need to give an explicit constant for Lemma~3.3 of~\cite{bp}, such that (using their notations, but numerating the constants $\th_1, \th_2,\dots$ in order to keep track of them in all steps)
\begin{equation}\label{lemma3.3}
	\A_2(\Om)\leq \th_1\left(\A(\Om_+)+\left|\frac12-\frac{|\Om_+|}{|\Om|}\right|+\A(\Om_-)+\left|\frac12-\frac{|\Om_-|}{|\Om|}\right|\right)^{2/3}.
\end{equation}
This constant is deduced by putting together three main inequalities.

$i)$ First of all we call $\eps_{\pm}=\frac12-\frac{|\Om_\pm|}{|\Om|}$ 
and call $B_\pm$ two balls centered in the origin such that $|\Om_{\pm}|=|B_{\pm}|$.
Hence there exist two points $x_{\pm}$ such that \[
\A(\Om_{\pm})=\frac{2|(B_\pm+x_\pm)\setminus \Om_\pm|}{|\Om_\pm|}.
\]
We then rescale the balls to measure $|\Om|/2$ each: $\widetilde B_{\pm}=(1-2\eps_\pm)^{-1/2}B_\pm$.
We have now to translate the new balls in the direction $x_+-x_-$ so that they are disjoint. It is easy to see that the width $l$ of the set $(\widetilde B_++x_+)\cap (\widetilde B_-+x_-)$ can be estimated by\[
l^{3/2}|\Om|^{1/4}\leq \th_2 |(\widetilde B_++x_+)\cap (\widetilde B_-+x_-)|, 
\] 
with the choice of $\th_2=(8\pi)^{1/4}$.

$ii)$ The second intermediate inequality is \[
|(\widetilde B_++x_+)\cap (\widetilde B_-+x_-)|\leq \th_3 |\Om|(\A(\Om_+)+\A(\Om_-)+|\eps_+|+|\eps_-|),
\]
and it is possible to see immediately that $\th_3=1$ works.

$iii)$ We now have to translate $\widetilde B_-$ so that it is tangent to $\widetilde B_++x_+$, and we will call $\overline x$ the new center. It is possible to prove that \[
|(\widetilde B_-+x_-)\cap (\widetilde B_-+\overline x)|\leq \th_4 l|\Om|^{1/2},
\] 
with the constant $\th_4=\frac{\pi+2}{(2\pi)^{1/2}}$.

At the end it is possible to put together the above inequalities and get~\eqref{lemma3.3} with $\th_1=\th_4(\th_2\th_3)^{2/3}+2=\frac{\pi+2}{\pi^{1/3}}+2.$

In order to conclude the proof of the quantitative Krahn-Szeg\"o inequality by Brasco and Pratelli we have to prove a last intermediate inequality, when $ks(\Om)\leq 1$:\[
ks(\Om)\geq \frac{1}{\th_5}\max{\left\{\A(\Om_+)^3+\left|\frac12-\frac{|\Om_+|}{|\Om|}\right|,\A(\Om_-)^3+\left|\frac12-\frac{|\Om_-|}{|\Om|}\right|\right\}},
\]
which works for $\th_5=\frac{3}{\betafk}$, where $\betafk$ is the constant \eqref{CQFK}  of the quantitative Faber-Krahn inequality that we computed in the paragraph above.
We note that in~\cite{bp} the exponent of the asymmetries is actually $4$ instead of $3$, since they use the quantitative Faber-Krahn~\eqref{QFK} instead of the one proved by Bhattacharya, which is only two dimensional.
On the other hand, if $ks(\Om)\geq 1$ it is enough to take $\th_5=2^3+1=9$ since $\A(\Om_{\pm})\leq 2$ and $|\eps_{\pm}|\leq 1/2$.
So we have that $\th_5=\max{\{9, 3/\betafk\}}=3/\betafk$.
Putting all the inequalities together one arrives to \[
\begin{split}
\betaks&=\frac{1}{\th_5}\,\frac{1}{\th_1^{9/2}}=\frac{\betafk}{3}\,\frac{\pi^{3/2}}{(\pi+2+2\pi^{1/3})^{9/2}}\\
&=\frac{\pi^{3/2}}{24(\pi+2+2\pi^{1/3})^{9/2}\cdot j_0^2}\cdot 10^{-5}\approx 3.331 \cdot 10^{-11}.
\end{split}
\]

\bigskip
\textbf{Acknowledgements.}
We thank Bozhidar Velichkov for pointing out a link between Theorem~\ref{teo.1} with \cite[Section 5.3]{bbv}.
The first author was supported by the ERC advanced grant n. 339958 ``COMPAT''.
The second author was supported by the INdAM-GNAMPA through the Project 2015 ``Critical
phenomena in the mechanics of materials: a variational approach''.

\end{document}